\documentclass[12pt]{article}
\usepackage[utf8]{inputenc}
\usepackage[a4paper]{geometry}
\usepackage{amsthm,amstext,amssymb,amsmath}
\usepackage{lmodern}
\usepackage{stmaryrd}
\usepackage{mathtools,graphicx,rotating}
\usepackage{cite}
\usepackage{tikz,booktabs}
\usetikzlibrary{matrix,shapes,calc}
\usepackage{lscape}
\usepackage{microtype}

\newtheorem{theorem}{Theorem}

\renewcommand{\le}{\leqslant}
\renewcommand{\ge}{\geqslant}

\newcommand{\emptysystem}{\mathbf{0}}
\newcommand{\symmdiff}{\triangle}
\newcommand{\MM}{{\cal M}}
\newcommand{\CC}{{\cal C}}
\newcommand{\sem}{{\mathrm{sem}}}
\renewcommand{\SS}{{\cal S}}
\newcommand{\FF}{{\cal F}}
\newcommand{\KK}{{\cal K}}
\renewcommand{\AA}{{\cal A}}
\newcommand{\NN}{\mathbf{N}}
\newcommand{\RR}{\mathbf{R}}
\newcommand{\abs}[1]{\left\lvert#1\right\rvert}
\newcommand{\norm}[1]{\left\lVert#1\right\rVert}
\newcommand{\unlab}[2]{\left\llbracket #1\right\rrbracket_{#2}}
\newcommand{\bigo}[1]{O\mathopen{}\left(#1\right)}
\newcommand{\sst}[2]{\left\{#1\,\middle|\,#2\right\}}
\newcommand{\intervalle}[4]{\mathopen{#1}#2\mathclose{}\mathpunct{},#3\mathclose{#4}}
\newcommand{\icc}[2]{\intervalle{[}{#1}{#2}{]}}

\tikzset{%
  vertex/.style={circle, draw=black, fill=black, inner sep=0.5pt, minimum size=5pt},
  root/.style={circle, draw=black, fill=black, inner sep=0.5pt, minimum size=5pt},
  edge/.style={very thick},%
  nedge/.style={very thick,dotted},%
}

\newenvironment{tikzgraph}
{\begin{tikzpicture}\begin{scope}}
  {\end{scope}\end{tikzpicture}}

\newcommand{\unit}{15pt}

\newcommand{\fourflag}{%
 \draw (-.6*\unit,.6*\unit) node[vertex] (a0) {};%
  \draw (.6*\unit,.6*\unit) node[vertex] (a1) {};%
  \draw (.6*\unit,-.6*\unit) node[vertex] (a2) {};%
  \draw (-.6*\unit,-.6*\unit) node[vertex] (a3) {};%
}
\newcommand{\ffrootednonedge}{%
       font=\scriptsize
 \draw (-.6*\unit,.6*\unit) node[vertex] (a0) {};%
  \draw (.6*\unit,.6*\unit) node[vertex] (a1) {};%
  \draw (.6*\unit,-.6*\unit) node[vertex] (a2) {};%
 \draw (.6*\unit,-.7*\unit) node[below] {$2$};%
  \draw (-.6*\unit,-.6*\unit) node[vertex] (a3) {};%
 \draw (-.6*\unit,-.7*\unit) node[below] {$1$};%
  \draw[nedge] (a2)--(a3);
}
\newcommand{\ffrootededge}{%
       font=\scriptsize
 \draw (-.6*\unit,.6*\unit) node[vertex] (a0) {};%
  \draw (.6*\unit,.6*\unit) node[vertex] (a1) {};%
  \draw (.6*\unit,-.6*\unit) node[vertex] (a2) {};%
 \draw (.6*\unit,-.7*\unit) node[below] {$2$};%
  \draw (-.6*\unit,-.6*\unit) node[vertex] (a3) {};%
 \draw (-.6*\unit,-.7*\unit) node[below] {$1$};%
  \draw[edge] (a2)--(a3);
}
\newcommand{\tfrootededge}{%
       font=\scriptsize
 \draw (0,.7*\unit) node[vertex] (a0) {};%
  \draw (-.6*\unit,-.3*\unit) node[vertex] (a1) {};%
  \draw (.6*\unit,-.3*\unit) node[vertex] (a2) {};%
 \draw (.6*\unit,-.4*\unit) node[below] {$2$};%
 \draw (-.6*\unit,-.4*\unit) node[below] {$1$};%
  \draw[edge] (a1)--(a2);
}

\begin{document}
\title{A new lower bound based on Gromov's method of selecting heavily covered
points\thanks{The work of the first and second authors leading to this
invention has received funding from the European Research Council under the
European Union's Seventh Framework Programme (FP7/2007-2013)/ERC grant
agreement no.~259385. The second author was also supported by the project
GIG/11/E023 of the Czech Science Foundation.
The work of the third author was partially supported by
the French \emph{Agence nationale de la recherche} under reference \textsc{anr
10 jcjc 0204 01} and by the \textsc{PHC Barrande 24444 XD}.}}
\author{Daniel Kr\'al'\thanks{Computer Science Institute, Faculty of Mathematics and Physics,
Charles University, Malostransk\'e n\'am\v est\'\i{} 25, 118 00 Prague 1,
Czech Republic. E-mail: \texttt{kral@iuuk.mff.cuni.cz}. Institute for Theoretical computer science is supported as project 1M0545 by Czech Ministry of Education.}
\and	Luk{\'a}{\v s} Mach\thanks{Computer Science Institute, Faculty
of Mathematics and Physics, Charles University, Malostransk\'e n\'am\v
est\'\i{} 25, 118 00 Prague 1, Czech Republic. E-mail: \texttt{lukas@iuuk.mff.cuni.cz}.}
\and	Jean-S{\'e}bastien Sereni\thanks{CNRS (LIAFA, Universit\'e Denis Diderot), Paris, France, and Department of Applied Mathematics (KAM), Faculty of Mathematics and Physics, Charles University, Prague, Czech Republic.  E-mail: \texttt{sereni@kam.mff.cuni.cz}.}
	}
\date{}
\maketitle
\begin{abstract}
Boros and F{\"u}redi (for $d=2$) and B{\'a}r{\'a}ny (for arbitrary $d$)
proved that there exists a positive real number $c_d$ such that
for every set $P$ of $n$ points in $\RR^d$ in general position, there exists
a point of $\RR^d$ contained in at least $c_d\binom{n}{d+1}$ $d$-simplices with vertices
at the points of $P$. Gromov improved the known lower bound on $c_d$ by
topological means.
Using methods from extremal combinatorics,
we improve one of the quantities appearing in Gromov's approach and
thereby provide a new stronger lower bound on $c_d$ for arbitrary $d$.
In particular, we improve the lower bound
on $c_3$ from $0.06332$ to more than $0.07480$; the best upper bound known on $c_3$
being $0.09375$.
\end{abstract}

\section{Introduction}
\label{sect-intro}
We study an extremal graph theory problem linked to a classical geometric
problem through a recent work of Gromov~\cite{Gro10a}. The geometric
result that initiated this work is a theorem of B{\'a}r{\'a}ny~\cite{Bar82},
which extends an earlier generalization of Carath{\'e}odory's theorem
due to Boros and F{\"u}redi~\cite{BoFu84}.
\begin{theorem}[B{\'a}r{\'a}ny~\cite{Bar82}]\label{th:geom}
  Let $d$ be a positive integer. There exists a positive real number $c$ such
  that for every set $P$ of points in $\RR^d$
  that are in general position, there is a point of
  $\RR^d$ that is contained in at least
  \begin{equation}\label{eq:geom}
  c\cdot\binom{\abs{P}}{d+1}-\bigo{\abs{P}^d}
\end{equation}
  $d$-dimensional simplices spanned by the points in $P$.
\end{theorem}
Define $c_d$ to be the supremum of all the real numbers that
satisfy~\eqref{eq:geom} in Theorem~\ref{th:geom} for the dimension $d$.

Bukh, Matou{\v s}ek and Nivasch~\cite{BMN10} established that
\[c_d\le\frac{(d+1)!}{(d+1)^{d+1}}\]
by constructing suitable configurations of $n$ points in $\RR^d$.
On the lower bound side, Boros and F{\"u}redi~\cite{BoFu84}
proved that $c_2\ge 2/9$,
which matches the upper bound; so $c_2=2/9$ (another proof
was given by Bukh~\cite{Buk06}).
B{\'a}r{\'a}ny's proof~\cite{Bar82} yields that $c_d\ge (d+1)^{-d}$.
Wagner~\cite{Wag03} improved this lower bound to
\[c_d\ge\frac{d^2+1}{(d+1)^{d+1}}.\]
Further improvements of the lower bound for $c_3$
were established by Basit et al.~\cite{BMRR10} and
by Matou\v{s}ek and Wagner~\cite{MaWa11}.

Gromov~\cite{Gro10a} developed a topological method
for establishing lower bounds on $c_d$
(Matou\v{s}ek and Wagner~\cite{MaWa11} provided
an exposition of the combinatorial components of his method, while
Karasev~\cite{Kar10} managed to simplify Gromov's approach).
His method yields a bound that matches the optimal bound
for $d=2$ and is better than that of Basit et al.~\cite{BMRR10}
for $d=3$.
We need several definitions to state Gromov's lower bound.
Fix a positive integer $d$ and a finite set $V$.
A \emph{$d$-system $E$ on $V$} is a family of $d$-element subsets of $V$.
The \emph{density} of the system $E$ is
$\norm{E}\coloneqq\abs{E}/{\binom{\abs{V}}{d}}$.
The \emph{coboundary $\delta E$} of a $d$-system $E$ on $V$
is the $(d+1)$-system composed of those $(d+1)$-element subsets of $V$
that contain an odd number of sets of $E$. The coboundary operator $\delta$
commutes with the symmetric difference, i.e.,
$\delta (A\symmdiff B)=(\delta A)\symmdiff(\delta B)$.
It is not hard to show that $\delta\delta E=\emptysystem$ for any $d$-system $E$
where $\emptysystem$ is the empty $(d+2)$-system. In fact, the converse also holds:
a $d$-system $E$ is a coboundary of a $(d-1)$-system if and only if $\delta E=\emptysystem$.

A $d$-system $E$ on $V$ is \emph{minimal} if $\norm{E}\le\norm{E'}$ for any $d$-system $E'$ on $V$
with $\delta E=\delta E'$. This is equivalent to saying that $\norm{E}\le\norm{E\symmdiff\delta D}$
for every $(d-1)$-system $D$ on $V$. Let $\MM_d(V)$ be the set of all minimal $d$-systems on $V$
and define the following function:
\[
  \varphi_d(\alpha)\coloneqq\liminf_{\abs{V}\to\infty}\min\sst{\norm{\delta
   E}}{\text{$E\in\MM_d(V)$ and $\norm{E}\ge\alpha$}}.
\]
It is easy to observe that the functions $\varphi_d$ are defined for $\alpha\in \icc{0}{1/2}$ and
$\varphi_1(\alpha)=2\alpha(1-\alpha)$. It can also be shown that $\varphi_d(\alpha)\ge\alpha$.

Gromov's lower bound on the quantity $c_d$ is given in the next theorem.
\begin{theorem}[Gromov~\cite{Gro10a}]\label{thm:gromlow}
  For every positive integer $d$, it holds that
\begin{equation}
c_d\ge\varphi_d\left(\frac{1}{2}\varphi_{d-1}\left(\frac{1}{3}\varphi_{d-2}\left(\cdots\frac{1}{d}\varphi_1\left(\frac{1}{d+1}\right)\cdots\right)\right)\right).
\label{eq:gromlow}
\end{equation}
\end{theorem}
Plugging $\varphi_1(\alpha)=2\alpha(1-\alpha)$ and
the bound $\varphi_d(\alpha)\ge\alpha$ in~\eqref{eq:gromlow},
we obtain
\begin{equation}
c_d\ge\frac{2d}{(d+1)!(d+1)}.
\label{eq:gromlow2}
\end{equation}

Improvements of the bound in~\eqref{eq:gromlow2} can be obtained by proving
stronger lower bounds on the functions $\varphi_d$.
The first step in this direction has been done by Matou{\v s}ek and Wagner.
\begin{theorem}[Matou{\v s}ek and Wagner~\cite{MaWa11}]\label{th:mawa}
  \mbox{}
  \begin{itemize}
    \item For all $\alpha\in\icc{0}{1/4}$, it holds that
\[
  \varphi_2(\alpha)\ge\frac{3}{4}\left(1-\sqrt{1-4\alpha}\right)(1-4\alpha).
\]
\item For all sufficiently small $\alpha>0$, it holds that
  \[
  \varphi_3(\alpha)\ge\frac{4}{3}\alpha-O(\alpha^2).
  \]
\end{itemize}
\end{theorem}
Our main result asserts a stronger lower bound on $\varphi_2(\alpha)$
for $\alpha\in\icc{0}{2/9}$, which are the values appearing in Theorem~\ref{thm:gromlow}.
\begin{theorem}
\label{thm:main}
For all $\alpha\in\icc{0}{2/9}$, it holds that
\[\varphi_2(\alpha)\ge\frac{3}{4}\alpha(3-\sqrt{8\alpha+1}).\]
\end{theorem}
When plugged into Theorem~\ref{thm:gromlow}, our bound yields that
$c_3>0.07433$.
For comparison, the
earlier bounds of Wagner~\cite{Wag03}, Basit et al.~\cite{BMRR10},
Gromov~\cite{Gro10a} and Matou\v{s}ek and Wagner~\cite{MaWa11}
are $c_3\ge0.03906$, $c_3\ge0.05448$, $c_3\ge0.0625$ and $c_3\ge0.06332$, respectively.
However, the bound on $c_3$ can be further improved as we now explain.

Matou\v sek and Wagner~\cite{MaWa11}
improved the bound on $c_3$ through a combinatorial argument,
which uses bounds on $\varphi_2$ and $\varphi_3$ as black-boxes.
The proof employs a structure called pagoda (of dimension 3) consisting of
a $4$-system $G$ (which is referred to as the \textit{top} of the pagoda),
$3$-systems $F_{ijk}$ (with $1\le i<j<k\le 4$),
$2$-systems $E_{ij}$ (with $1\le i<j\le 4$) and $1$-systems $V_i$ (with $1\le i\le 4$).
For a precise definition of these sets and their interplay, we refer the reader to \cite[Section 6]{MaWa11}.
Any lower bound on the density of $G$ in a pagoda is also a lower bound on $c_3$.
Gromov's approach is applicable to pagodas and
it yields that $||G|| \ge \frac{1}{16} = 0.0625$ (using the trivial bounds on $\varphi_2$ and $\varphi_3$).
Matou\v sek and Wagner investigated pagodas with $||G|| = 0.0625 + \varepsilon$
and they obtain a contradiction for $\varepsilon \le 0.00082$; this proves that $||G|| \ge 0.06332$.

We can improve the bound using our Theorem~\ref{thm:main} on $\varphi_2$
by investigating pagodas with density $3(3-\sqrt{2})/64+\varepsilon$.
This leads to the following system of inequalities:
\begin{eqnarray*}
\varphi_2\left(0.125 + 2 \cdot \varepsilon_0\right) & \le & 2 \cdot \left(\frac{3(3-\sqrt{2})}{64} + \varepsilon\right) \\
\varphi_1\left(0.25 + \varepsilon_1\right) & \le & 3 \cdot \left(0.125 + 2\varepsilon_0\right) \\
4 \cdot \varphi_1\left(0.25 - 3\varepsilon_1\right) & \ge & 4\left(\frac{3}{8} - \varepsilon_2\right) \\
2 \cdot \left(\frac{3(3-\sqrt{2})}{64} + \varepsilon\right) & \ge & -6\varepsilon_0 + 24\varepsilon_1^2 + 2\varepsilon_1\varepsilon_2 - \frac{27}{4}\varepsilon_1 - \frac{3}{2}\varepsilon_2 + \frac{3}{16}
\end{eqnarray*}
This system of inequalities together with
the exact value of $\varphi_1$, Theorem~\ref{thm:main} and
the trivial (linear) bound on $\varphi_3$ yields
a contradiction for every $\varepsilon \le 0.00047$.
This leads to the lower bound $c_3 \ge 0.07480$.

The definition of the function $\varphi_2$ can naturally be cast
in the language of graphs.
A \emph{cut} of a graph $G$ is a partition
of the vertices of $G$ into two (disjoint) parts; a (non-)edge
that crosses the partition is said to be \emph{contained} in the cut.
A graph is \emph{Seidel-minimal} if no cut contains more edges
than non-edges. It is straightforward to see that a graph $G$ with vertex set $V$
is Seidel-minimal if and only if its edge-set viewed as a $2$-system
is minimal. Let $\SS_n(\alpha)$ be the set of all Seidel-minimal graphs
on $n$ vertices with density at least $\alpha$, i.e.,
with at least $\alpha\binom{n}{2}$ edges. Further,
let $\SS(\alpha)$ be the union of all $\SS_n(\alpha)$.

A triple $T$ of vertices of a graph $G$ is \emph{odd} if the subgraph of $G$
induced by $T$ contains precisely either one or three edges.
Finally, let $\varphi_g(G)$ for a graph $G$ be the density of odd triples in $G$,
i.e.,
\[\varphi_g(G)=\frac{\abs{\sst{T\in \binom{V(G)}{3}}{\text{$T$ is
odd}}}}{\binom{\abs{V(G)}}{3}}.\]
It is not hard to show that for every $\alpha\in\icc{0}{1/2}$,
\[\varphi_2(\alpha)=\liminf_{n\to\infty}\min\sst{\varphi_g(G)}{G\in\SS_n(\alpha)}.\]

Using this reformulation to the language of graph theory, we show that
$\varphi_2(\alpha)\ge\frac{3}{4}\alpha(3-\sqrt{8\alpha+1})$ for $\alpha\in \icc{0}{2/9}$.
Our proof is based on the notion of flag algebras developed by
Razborov~\cite{Raz07}, which builds on the work of Lovász and
Szegedy~\cite{LoSz06} on graph limits and of Freedman et al.~\cite{FLS07}.
The notion was further applied, e.g.,
in~\cite{BaTa11,Grz11,HHK+11a,HHK+11b,HKN09,Raz10b}.
We do not use the full strength of this notion here and we survey
the relevant parts in Section~\ref{sect-flag} to make the paper as much
self-contained as possible. In Section~\ref{sect-first}, we establish
a weaker bound $\varphi_2(\alpha)\ge\frac{9}{7}\alpha(1-\alpha)$ using just some
of the methods presented in Section~\ref{sect-flag}. The purpose of
Section~\ref{sect-first} is to get the reader acquainted with the notation.
Our main result is proved in Section~\ref{sect-second}.

\section{Flag algebras}
\label{sect-flag}
In this section, we review some of the theory related to flag algebras,
which were introduced by Razborov~\cite{Raz07}. We focus on the concepts that
are relevant to our proof.
The reader is referred to the seminal paper of Razborov~\cite{Raz07}
for a complete and detailed exposition of the topic.

Fix $\alpha>0$ and consider a sequence of graphs $(G_i)_{i\in\NN}$ from $\SS(\alpha)$ such that
\[
  \lim_{i\to\infty} \abs{V(G_i)}=\infty\quad\text{and}\quad\lim_{i\to\infty}\varphi_g(G_i)=\varphi_2(\alpha).
\]
Let $p(H,H_0)$ be the probability that a randomly chosen subgraph of $H_0$
with $\abs{V(H)}$ vertices is isomorphic to $H$. The sequence $G_i$ must contain
a subsequence $(G_{i_j})_{j\in\NN}$ such that $\lim_{j\to\infty}p(H,G_{i_j})$
exists for every graph $H$. Define $q_\alpha(H)\coloneqq\lim_{j\to\infty}p(H,G_{i_j})$.
Observe that the definition of $q_\alpha$ implies that $q_\alpha\left(K_2\right)\ge\alpha$ and
$q_\alpha\left(\overline{P_3}\right)+q_\alpha\left(K_3\right)=\varphi_2(\alpha)$
where $\overline{P_3}$
is the complement of the $3$-vertex path.

The values of $q_\alpha\left(H\right)$ for various graphs $H$ are highly correlated.
Let $\FF$ be the set of all graphs and $\FF_\ell$ the set of graphs with $\ell$ vertices.
Extend the mapping $q_\alpha\left(H\right)$ from $\FF$ to $\RR\FF$ by
linearity,
where $\RR\FF$ is the linear space of formal linear combinations of the elements of $\FF$
with real coefficients.
Next, let $\KK$ be the subspace of $\RR\FF$
generated by the elements of the form \[H_0-\sum_{H\in\FF_\ell}p(H_0,H)H\]
for all graphs $H_0$ and all $\ell>\abs{V(H_0)}$.
Since the quantity $p(H_0,G)$ and the sum $\sum_{H\in\FF_\ell}p(H_0,H)p(H,G)$
are equal for any graph $G$ with at least $\ell$ vertices,
$\KK$ is a subset of the kernel of $q_\alpha$,
i.e., $q_\alpha\left(F\right)=q_\alpha\left(F+F'\right)$ for every $F\in\RR\FF$ and $F'\in\KK$.

Let $p(H_1,H_2;H_0)$ be the probability that two randomly chosen disjoint subsets $V_1$ and $V_2$
with cardinalities $\abs{V(H_1)}$ and $\abs{V(H_2)}$ induce in $H_0$ subgraphs isomorphic to $H_1$ and
$H_2$, respectively.
For two graphs $H_1$ and $H_2$, define their product to be
\[H_1\times H_2\coloneqq\sum_{H_0\in\FF_\ell}p(H_1,H_2;H_0)H_0\]
where $\ell=\abs{V(H_1)}+\abs{V(H_2)}$.
The product operator can be extended to $\RR\FF\times\RR\FF$ by linearity.
Since the product operator defined in this way is consistent with the equivalence relation
on the elements of $\RR\FF$ induced by $\KK$, we can consider the quotient $\AA\coloneqq\RR\FF/\KK$
as an algebra with addition and multiplication.
Since $q_\alpha$ is consistent with $\KK$, the function $q_\alpha$ naturally gives rise
to a mapping from $\AA$ to $\RR$, which is in fact a homomorphism from $\AA$ to $\RR$.
In what follows, we use $q_\alpha$ for this homomorphism exclusively. To simplify
our notation, we will use $q_\alpha\left(F\right)$ for $F\in\RR\FF$ but
we also keep in mind that
$F$ stands for a representative of the equivalence class of $\RR\FF/\KK$.

A homomorphism $q:\AA\to\RR$ is \emph{positive} if $q(F)\ge 0$
for every $F\in\FF$. Positive homomorphisms are precisely those corresponding
to the limits of convergent graph sequences.
We write $F\ge 0$ for $F\in\AA$ if $q(F)\ge 0$ for any positive homomorphism
$q$. Such $F\in\AA$ form the semantic cone $\CC_\sem(\AA)$.
Razborov~\cite{Raz07} developed various general and deep methods
for proving that $F\ge 0$ for $F\in\AA$. Here, we will use only one of them,
which we now present. The reader may also check the paper~\cite{Raz10b}
for the exposition of the method in a more specific context.

Consider a graph $\sigma$ and let $\FF^\sigma$
be the set of graphs $G$ equipped with a mapping $\nu:\sigma\to V(G)$
such that $\nu$ is an embedding of $\sigma$ in $G$, i.e.,
the subgraph induced by the image of $\nu$ is isomorphic to $\sigma$.
We can extend the definitions of the quantities $p(H,H_0)$ and $p(H_1,H_2;H_0)$
to this ``labeled'' case by requiring that the randomly chosen sets
always include the image of $\nu$ and preserve the mapping $\nu$.
In particular, $p(H_1,H_2;H_0)$ is the probability that two randomly chosen
supersets of the image of $\sigma$ in $H_0$ with sizes $V(H_1)$ and $V(H_2)$ that
intersect exactly on $\sigma$ induce subgraphs of $H_0$ isomorphic to $H_1$ and $H_2$;
Similarly as before, one can define $\KK^\sigma$, $\AA^\sigma=\RR\FF^\sigma/\KK^\sigma$ as
an algebra with addition and multiplication, positive homomorphisms, etc.

The intuitive interpretation of homomorphisms from $\AA^\sigma$ to $\RR$
is as follows: for a fixed embedding $\nu$ of $\sigma$,
the value $q_\nu(F)$ for $F\in\FF^\sigma$ is the probability that
a randomly chosen superset of the image of $\nu$
induces a subgraph isomorphic to $F$.
A positive homomorphism $q$ from $\AA$ to $\RR$
gives rise to a unique probability distribution
on positive homomorphisms $q^\sigma$
from $\AA^\sigma$ to $\RR$ such that this probability distribution
is the limit of the probability distributions of homomorphisms $q_\nu$ from $\AA^\sigma$ to $\RR$
given by random choices of $\nu$ in the graphs in any convergent sequence corresponding to $q$,
see~\cite[Section 3.2]{Raz07} for details.

Consider a graph $H$ with an embedding $\nu$ of $\sigma$ in $G$.
Define $\unlab{H}{\sigma}$ to be the element $p\cdot H$ of $\AA$
where $p$ is the probability that
a randomly chosen mapping $\nu$ from $V(\sigma)$ to $V(H)$ is an embedding of $\sigma$ in $H$.
Hence, the operator $\unlab{\cdot}{\sigma}$ maps elements of $\FF^\sigma$ to $\AA$ and
it can be extended from $\FF^\sigma$ to $\AA^\sigma$ by linearity.
For a positive homomorphism $q$ from $\AA$ to $\RR$,
the value of $q(\unlab{H}\sigma)$ for $H\in\AA^\sigma$ is
the expected value of $q^\sigma(H)$ with respect to the probability distribution
on $q^\sigma$ corresponding to $q$.
In particular, if $q^\sigma(H)\ge 0$ with probability one, then $q(\unlab{H}\sigma)\ge 0$.

\subsection{Example}
\label{ssect-example}
As an example of the introduced formalism, we prove that $\varphi_2(\alpha)\ge\alpha$.
The following notation is used: $K_n$ is the complete graph with $n$ vertices,
$P_n$ is the $n$-vertex path and $\overline{K}_n$ and $\overline{P}_n$ are
their complements, respectively.
We also use $1$ for $K_1$ to simplify the notation. The following elements of $\AA^1$
will be of particular interest to us:
$P_3^{1,b}$ is $P_3$ with $1$ embedded to the end vertex of the path and
$P_3^{1,c}$ is $P_3$ with $1$ embedded to the central vertex;
$\overline{P_3}^{1,b}$ and $\overline{P_3}^{1,c}$ are their complements,
respectively.
See Figure~\ref{fig:ill} for an illustration of this notation.

\begin{figure}
    \begin{center}
        \begin{tikzgraph}
\matrix[column sep=4em, row sep=1em] {
 \draw[fill=black!30] (0,0) circle (3mm);
 \draw[edge] (0,0) node[vertex] (a0) {} -- (.7,1) node[vertex]
 (a1) {} -- (1.4,0) node[vertex] (a1) {};
  \draw[nedge] (a0)--(a1);
 \node at (.7,-1) () {$P_3^{1,b}$}; &
\draw[fill=black!30] (.7,1) circle (3mm);
 \draw[edge] (0,0) node[vertex] (a0) {} -- (.7,1) node[vertex]
 (a1) {} -- (1.4,0) node[vertex] (a1) {};
  \draw[nedge] (a0)--(a1);
\node at (.7,-1) () {$P_3^{1,c}$}; &
\draw[fill=black!30] (0,0) circle (3mm);
 \draw[edge] (0,0) node[vertex] (a0) {};
 \draw[edge] (.7,1) node[vertex] (a1) {};
 \draw[edge] (1.4,0) node[vertex] (a2) {};
 \draw[nedge] (a1)--(a2);
 \draw[nedge] (a1)--(a0);
 \draw[edge] (a0)--(a2);
 \node at (.7,-1) () {$\overline{P_3}^{1,b}$}; &
\draw[fill=black!30] (.7,1) circle (3mm);
 \draw[edge] (0,0) node[vertex] (a0) {};
 \draw[edge] (.7,1) node[vertex] (a1) {};
 \draw[edge] (1.4,0) node[vertex] (a2) {};
 \draw[nedge] (a1)--(a2);
 \draw[nedge] (a1)--(a0);
 \draw[edge] (a0)--(a2);
\node at (.7,-1) () {$\overline{P_3}^{1,c}$}; \\
};
        \end{tikzgraph}
    \end{center}
\caption{Four elements of $\AA^1$.}\label{fig:ill}
\end{figure}

Consider the homomorphism $q_\alpha$ from $\AA$ to $\RR$.
Recall that $\alpha\le q_\alpha\left(K_2\right)$.
Since it holds that
\[K_2-\frac{1}{3}\overline{P_3}-\frac{2}{3}P_3-K_3\in\KK,\]
we obtain
\begin{equation}
\alpha\le q_\alpha\left(\frac{1}{3}\overline{P_3}+\frac{2}{3}P_3+K_3\right).\label{ex1}
\end{equation}
We now use that the graphs in the sequence defining $q_\alpha$ are Seidel-minimal.
Let $G_i$ be a graph in this sequence, $n$ the number of its vertices and $v$ an arbitrary vertex
of $G_i$.
Let $A$ be the neighbors of $v$ and $B$ its non-neighbors.
Since $G$ is Seidel-minimal, the number of edges between $A$ and $B$
does not exceed the number of non-edges between $A$ and $B$
(increased by $\bigo{n}$ for the inclusion of $v$ in one or the other side of the cut;
however, this term will vanish in the limit). So, if $\sigma=1$ is an embedding of $K_1$,
it holds that $q_\alpha^1(\overline{P_3}^{1,b}-P_3^{1,b})\ge 0$ with probability one (the term
$q_\alpha^1(\overline{P_3}^{1,b})$ represents the number of non-edges between neighbors and
non-neighbors of the target vertex of $\sigma$ and $q_\alpha^1(P_3^{1,b})$ the number of edges).
Therefore, we obtain
\begin{equation}
0\le q_\alpha\left(\unlab{\overline{P_3}^{1,b}-P_3^{1,b}}{1}\right).\label{ex2}
\end{equation}
Applying the operator $\unlab{\cdot}{1}$ in~\eqref{ex2} yields that
\begin{equation}
0\le q_\alpha\left(\frac{2}{3}\overline{P_3}-\frac{2}{3}P_3\right).\label{ex3}
\end{equation}
Summing~\eqref{ex1} and~\eqref{ex3} (recall that $q_\alpha$ is a homomorphism
from $\AA$ to $\RR$), we obtain
\[\alpha\le q_\alpha\left(\overline{P_3}+K_3\right)=\varphi_2(\alpha).\]
This completes the proof.

A similar argument applied to the algebra based on $d$-uniform hypergraphs
yields that $\varphi_d(\alpha)\ge\alpha$. However, since we do not want to introduce
additional notation not necessary for the exposition in the rest of the paper,
we omit further details.

\section{First bound}
\label{sect-first}
To become more acquainted with the method, we now present a bound that is
both weaker and simpler than our main result.
Fix the enumeration of $4$-vertex graphs as in Figure~\ref{fig-four}.
To simplify our formulas, $q_\alpha\left(\sum_{i=1}^{11}\xi_i F_i\right)$
shall simply be written $q_\alpha\left(\xi_1,\cdots,\xi_{11}\right)$.

\begin{figure}
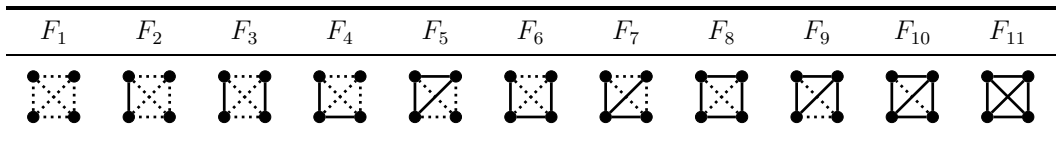

\begin{center}
  \scalebox{.85}{
  \begin{tabular}{ccccccccccc}
    \toprule
$F_1$&$F_2$&$F_3$&$F_4$&$F_5$&$F_6$&$F_7$&$F_8$&$F_9$&$F_{10}$&$F_{11}$\\
    \midrule
  \begin{tikzgraph}
\fourflag
  \draw[nedge] (a3)--(a0);
  \draw[nedge] (a0)--(a2);
  \draw[nedge] (a1)--(a0);
  \draw[nedge] (a1)--(a3);
  \draw[nedge] (a2)--(a3);
  \draw[nedge] (a2)--(a1);
    \useasboundingbox (-1*\unit,-1*\unit) rectangle (1*\unit,1*\unit);
  \end{tikzgraph}
&
\begin{tikzgraph}
    \fourflag
    \draw[edge] (a0)--(a3);
  \draw[nedge] (a0)--(a2);
  \draw[nedge] (a1)--(a0);
  \draw[nedge] (a1)--(a3);
  \draw[nedge] (a2)--(a3);
  \draw[nedge] (a2)--(a1);
    \useasboundingbox (-1*\unit,-1*\unit) rectangle (1*\unit,1*\unit);
  \end{tikzgraph}
&
\begin{tikzgraph}
  \fourflag
  \draw[edge] (a0)--(a3);
  \draw[edge] (a1)--(a2);
  \draw[nedge] (a0)--(a2);
  \draw[nedge] (a1)--(a0);
  \draw[nedge] (a1)--(a3);
  \draw[nedge] (a2)--(a3);
    \useasboundingbox (-1*\unit,-1*\unit) rectangle (1*\unit,1*\unit);
\end{tikzgraph}
&
\begin{tikzgraph}
  \fourflag
  \draw[edge] (a0)--(a3);
  \draw[edge] (a3)--(a2);
  \draw[nedge] (a0)--(a2);
  \draw[nedge] (a1)--(a0);
  \draw[nedge] (a2)--(a1);
  \draw[nedge] (a3)--(a1);
    \useasboundingbox (-1*\unit,-1*\unit) rectangle (1*\unit,1*\unit);
\end{tikzgraph}
&
\begin{tikzgraph}
  \fourflag
  \draw[edge] (a0)--(a3);
  \draw[edge] (a0)--(a1);
  \draw[edge] (a1)--(a3);
  \draw[nedge] (a2)--(a0);
  \draw[nedge] (a2)--(a1);
  \draw[nedge] (a3)--(a2);
    \useasboundingbox (-1*\unit,-1*\unit) rectangle (1*\unit,1*\unit);
\end{tikzgraph}
&
\begin{tikzgraph}
  \fourflag
  \draw[edge] (a0)--(a3)--(a2)--(a1);
  \draw[nedge] (a0)--(a2);
  \draw[nedge] (a1)--(a3);
  \draw[nedge] (a1)--(a0);
    \useasboundingbox (-1*\unit,-1*\unit) rectangle (1*\unit,1*\unit);
\end{tikzgraph}
&
\begin{tikzgraph}
  \fourflag
  \draw[edge] (a0)--(a3);
  \draw[edge] (a3)--(a2);
  \draw[edge] (a1)--(a3);
  \draw[nedge] (a0)--(a2);
  \draw[nedge] (a2)--(a1);
  \draw[nedge] (a1)--(a0);
    \useasboundingbox (-1*\unit,-1*\unit) rectangle (1*\unit,1*\unit);
\end{tikzgraph}
&
\begin{tikzgraph}
  \fourflag
  \draw[edge] (a0)--(a3);
  \draw[edge] (a3)--(a2);
  \draw[edge] (a1)--(a2);
  \draw[edge] (a1)--(a0);
  \draw[nedge] (a0)--(a2);
  \draw[nedge] (a1)--(a3);
    \useasboundingbox (-1*\unit,-1*\unit) rectangle (1*\unit,1*\unit);
\end{tikzgraph}
&
\begin{tikzgraph}
  \fourflag
  \draw[edge] (a0)--(a3);
  \draw[edge] (a3)--(a1);
  \draw[edge] (a1)--(a0);
  \draw[edge] (a1)--(a2);
  \draw[nedge] (a0)--(a2);
  \draw[nedge] (a2)--(a3);
    \useasboundingbox (-1*\unit,-1*\unit) rectangle (1*\unit,1*\unit);
\end{tikzgraph}
&
\begin{tikzgraph}
  \fourflag
  \draw[edge] (a0)--(a3);
  \draw[edge] (a3)--(a2);
  \draw[edge] (a3)--(a1);
  \draw[edge] (a1)--(a2);
  \draw[edge] (a1)--(a0);
  \draw[nedge] (a2)--(a0);
    \useasboundingbox (-1*\unit,-1*\unit) rectangle (1*\unit,1*\unit);
\end{tikzgraph}
&
\begin{tikzgraph}
  \fourflag
  \draw[edge] (a0)--(a3);
  \draw[edge] (a3)--(a2);
  \draw[edge] (a3)--(a1);
  \draw[edge] (a1)--(a2);
  \draw[edge] (a1)--(a0);
  \draw[edge] (a2)--(a0);
    \useasboundingbox (-1*\unit,-1*\unit) rectangle (1*\unit,1*\unit);
  \end{tikzgraph}
\\
\bottomrule
\end{tabular}}
\end{center}
\caption{The eleven non-isomorphic graphs with $4$ vertices.}\label{fig-four}
\end{figure}


\begin{theorem}
\label{thm-first}
For every $\alpha\in \icc{0}{2/9}$, it holds that $q_\alpha\left(\overline{P_3}+K_3\right)\ge\frac{9}{7}\alpha(1-\alpha)$.
\end{theorem}

\begin{proof}
We first establish three inequalities on the values taken by $q_\alpha$
for various elements of $\AA$.
The choice of the graphs in the sequence defining $q_\alpha$ implies that
$\alpha\le q_\alpha\left(K_2\right)$.
As $q_\alpha\left(\overline{K}_2\right)=1-q_\alpha\left(K_2\right)$ and
$q_\alpha(K_2)\in \icc{0}{1/2}$, we infer that
\begin{equation}
\alpha(1-\alpha)\le q_\alpha\left(K_2\right)q_\alpha\left(\overline{K}_2\right)=
                    q_\alpha\left(K_2\times\overline{K}_2\right)=
		    q_\alpha\left(0,\frac{1}{6},0,\frac{1}{3},\frac{1}{2},\frac{1}{6},\frac{1}{2},0,\frac{1}{3},\frac{1}{6},0\right).\label{first2}
\end{equation}
The other two inequalities follow from the Seidel minimality of graphs
in the sequence defining $q_\alpha$. Consider a graph $G_i$ and
two non-adjacent vertices $v_1$ and $v_2$ (the target vertices of
an embedding of $\overline{K_2}$ in elements of $\FF^{\overline{K_2}}$
are marked by the numbers $1$ and $2$).
Let $A$ be the set of their common
neighbors and $B$ the set of the remaining vertices.
Applying the Seidel minimality to the cut given by $A$ and $B$,
we obtain the following inequality in the limit (the elements of $\FF^{\overline{K_2}}_4$
with a non-edge between a common neighbor of $1$ and $2$ and a vertex that is
not their common neighbor appear with the coefficient $+1$, those with an edge
between two such vertices with the coefficient $-1$).
\[
0\le q_\alpha\left(\unlab{%
\begin{tikzpicture}[baseline=-.8ex]
    \ffrootednonedge
    \draw[edge] (a2)--(a1);
    \draw[edge] (a3)--(a1);
    \draw[nedge] (a3)--(a0);
    \draw[nedge] (a2)--(a0);
    \draw[nedge] (a1)--(a0);
    font=\normalsize
    \node at (1.2*\unit,0) {$+$};
  \begin{scope}[xshift=2.5*\unit]
    \ffrootednonedge
    \draw[edge] (a2)--(a1);
    \draw[edge] (a3)--(a1);
    \draw[edge] (a3)--(a0);
    \draw[nedge] (a2)--(a0);
    \draw[nedge] (a1)--(a0);
    \end{scope}
  \node at (3.8*\unit,0) {$+$};
  \begin{scope}[xshift=5*\unit]
      \ffrootednonedge
    \draw[edge] (a2)--(a1);
    \draw[edge] (a3)--(a1);
    \draw[edge] (a2)--(a0);
    \draw[nedge] (a3)--(a0);
    \draw[nedge] (a1)--(a0);
    \end{scope}
  \node at (6.2*\unit,0) {$-$};
  \begin{scope}[xshift=7.5*\unit]
      \ffrootednonedge
    \draw[edge] (a2)--(a1);
    \draw[edge] (a3)--(a1);
    \draw[edge] (a1)--(a0);
    \draw[nedge] (a2)--(a0);
    \draw[nedge] (a3)--(a0);
    \end{scope}
  \node at (8.7*\unit,0) {$-$};
  \begin{scope}[xshift=10*\unit]
    \ffrootednonedge
    \draw[edge] (a2)--(a1);
    \draw[edge] (a3)--(a1);
    \draw[edge] (a3)--(a0);
    \draw[edge] (a1)--(a0);
    \draw[nedge] (a2)--(a0);
    \end{scope}
  \node at (11.2*\unit,0) {$-$};
  \begin{scope}[xshift=12.5*\unit]
      \ffrootednonedge
    \draw[edge] (a2)--(a1);
    \draw[edge] (a3)--(a1);
    \draw[edge] (a2)--(a0);
    \draw[edge] (a1)--(a0);
    \draw[nedge] (a3)--(a0);
    \end{scope}
  \end{tikzpicture}}{\overline{K}_2}
               \right).
  \]
Evaluating the operator $\unlab{\cdot}{\overline{K}_2}$ yields that
\begin{equation}
0\le q_\alpha\left(0,0,0,\frac{1}{6},0,\frac{1}{3},-\frac{1}{2},0,-\frac{1}{3},0,0
             \right).\label{first3}
\end{equation}
Now, let $A'$ be the neighbors of $v_2$ and $B'$ its non-neighbors.
The Seidel minimality of cuts of this type (the elements of $\FF^{\overline{K_2}}_4$ with a non-edge between a neighbor of $2$ and
a non-neighbor of $2$ appear with the coefficient $+1$ and those with an edge with the coefficient $-1$) yields that
\[
0\le q_\alpha\left(\unlab{%
\begin{tikzpicture}[baseline=-.8ex]
    \ffrootednonedge
    \draw[edge] (a2)--(a1);
    \draw[nedge] (a3)--(a1);
    \draw[nedge] (a3)--(a0);
    \draw[nedge] (a2)--(a0);
    \draw[nedge] (a1)--(a0);
    font=\normalsize
    \node at (1.2*\unit,0) {$+$};
  \begin{scope}[xshift=2.5*\unit]
    \ffrootednonedge
    \draw[edge] (a2)--(a1);
    \draw[edge] (a3)--(a1);
    \draw[nedge] (a3)--(a0);
    \draw[nedge] (a2)--(a0);
    \draw[nedge] (a1)--(a0);
    \end{scope}
  \node at (3.8*\unit,0) {$+$};
  \begin{scope}[xshift=5*\unit]
      \ffrootednonedge
    \draw[edge] (a2)--(a1);
    \draw[edge] (a3)--(a0);
    \draw[nedge] (a3)--(a1);
    \draw[nedge] (a2)--(a0);
    \draw[nedge] (a1)--(a0);
    \end{scope}
  \node at (6.2*\unit,0) {$+$};
  \begin{scope}[xshift=7.5*\unit]
      \ffrootednonedge
    \draw[edge] (a2)--(a1);
    \draw[edge] (a3)--(a1);
    \draw[edge] (a3)--(a0);
    \draw[nedge] (a2)--(a0);
    \draw[nedge] (a1)--(a0);
    \end{scope}
  \node at (8.7*\unit,0) {$-$};
  \begin{scope}[xshift=10*\unit]
    \ffrootednonedge
    \draw[edge] (a1)--(a0);
    \draw[edge] (a2)--(a1);
    \draw[nedge] (a3)--(a1);
    \draw[nedge] (a3)--(a0);
    \draw[nedge] (a2)--(a0);
    \end{scope}
  \node at (11.2*\unit,0) {$-$};
  \begin{scope}[xshift=12.5*\unit]
      \ffrootednonedge
    \draw[edge] (a1)--(a0);
    \draw[edge] (a2)--(a1);
    \draw[edge] (a3)--(a1);
    \draw[nedge] (a3)--(a0);
    \draw[nedge] (a2)--(a0);
    \end{scope}
  \node at (13.6*\unit,0) {$-$};
  \begin{scope}[xshift=14.9*\unit]
      \ffrootednonedge
    \draw[edge] (a1)--(a0);
    \draw[edge] (a2)--(a1);
    \draw[edge] (a3)--(a0);
    \draw[nedge] (a3)--(a1);
    \draw[nedge] (a2)--(a0);
    \end{scope}
  \node at (16.1*\unit,0) {$-$};
  \begin{scope}[xshift=17.4*\unit]
      \ffrootednonedge
    \draw[edge] (a1)--(a0);
    \draw[edge] (a2)--(a1);
    \draw[edge] (a3)--(a1);
    \draw[edge] (a3)--(a0);
    \draw[nedge] (a2)--(a0);
    \end{scope}
  \end{tikzpicture}}{\overline{K}_2}
               \right),
\]
which subsequently implies that
\begin{equation}
0\le q_\alpha\left(0,\frac{1}{3},\frac{2}{3},0,0,0,-\frac{1}{2},0,-\frac{1}{6},0,0
             \right).\label{first4}
\end{equation}
The sum of~\eqref{first2}, \eqref{first3} and~\eqref{first4} with coefficients $9/7$, $3/7$ and $6/7$
is the following inequality:
\begin{equation}
\frac{9}{7}\alpha(1-\alpha)
 \le q_\alpha\left(0,\frac{1}{2},\frac{4}{7},\frac{1}{2},\frac{9}{14},
                   \frac{5}{14},0,0,\frac{1}{7},\frac{3}{14},0
             \right).\label{first0}
\end{equation}
Since $q_\alpha$ is positive, we infer from~\eqref{first0} that
\[\frac{9}{7}\alpha(1-\alpha)\le q_\alpha\left(0,\frac{1}{2},1,\frac{1}{2},1,\frac{1}{2},0,0,\frac{1}{2},\frac{1}{2},1\right)=q_\alpha\left(\overline{P_3}+K_3\right).\]
\end{proof}

\section{Improved bound}
\label{sect-second}
This section is devoted to the proof of Theorem~\ref{thm:main}. We equivalently
prove the following.
\begin{theorem}
For every $\alpha\in \icc{0}{2/9}$, it holds that $q_\alpha\left(\overline{P_3}+K_3\right)\ge\frac{3}{4}\alpha(3-\sqrt{8\alpha+1})$.
\end{theorem}
\begin{proof}
Let $\beta\coloneqq q_\alpha\left(K_2\right)$. Note that $\beta\in\icc{\alpha}{1/2}$. We first derive two equalities using
the fact that $q_\alpha$ is a homomorphism from $\AA$ to $\RR$.
The first equation is a trivial corollary of this fact.
\begin{equation}
1=q_\alpha\left(K_1\right)=q_\alpha\left(1,1,1,1,1, 1,1,1,1,1, 1\right).
\label{secAA}
\end{equation}
The choice of $\beta$ implies that
\begin{equation}
\beta=q_\alpha\left(K_2\right)
=q_\alpha\left(0,\frac{1}{6},\frac{1}{3},\frac{1}{3},\frac{1}{2},\frac{1}{2},\frac{1}{2},\frac{2}{3},\frac{2}{3},\frac{5}{6},1\right).
\label{secA}
\end{equation}
The next equality is little bit more tricky. We use that $q_\alpha\left(K_2\right)-\beta=0$.
\begin{equation}
0=(q_\alpha\left(K_2\right)-\beta)q_\alpha\left(\overline{K}_2\right)=q_\alpha\left(K_2\times\overline{K}_2-\beta\overline{K}_2\right).
\label{sec2}
\end{equation}
Again, we express~\eqref{sec2} in terms of the four-vertex graphs:
\begin{equation}
0=q_\alpha\left(-\beta,\frac{1-5\beta}{6},\frac{-2\beta}{3},\frac{1-2\beta}{3},\frac{1-\beta}{2},
             \frac{1-3\beta}{6},\frac{1-\beta}{2},\frac{-\beta}{3},\frac{1-\beta}{3},\frac{1-\beta}{6},
	     0\right).
\label{secB}
\end{equation}
The next inequality is the inequality~\eqref{first4} established in the proof of Theorem~\ref{thm-first}.
We copy the inequality to ease the reading.
\begin{equation}\label{secC}
0\le q_\alpha\left(0,0,0,\frac{1}{6},0,\frac{1}{3},-\frac{1}{2},0,-\frac{1}{3},0,0\right).
\end{equation}
The final inequality is obtained by considering random homomorphisms $q_\alpha^{K_2}$.
Since $q_\alpha^{K_2}$ is a homomorphism, it holds for every choice of $q_\alpha^{K_2}$ and every $\xi\in\RR$ that
\begin{align}
0&\le q_\alpha^{K_2}\left(
\begin{tikzpicture}[baseline=-.8ex]
  \tfrootededge
    \draw[nedge] (a2)--(a0);
    \draw[nedge] (a1)--(a0);
    font=\normalsize
    \node at (1.8*\unit,0) {$-\xi\times$};
  \begin{scope}[xshift=3.5*\unit]
    \tfrootededge
    \draw[edge] (a1)--(a0);
    \draw[nedge] (a2)--(a0);
    \end{scope}
  \node at (5.4*\unit,0) {$-\xi\times$};
  \begin{scope}[xshift=7.1*\unit]
      \tfrootededge
    \draw[edge] (a2)--(a0);
    \draw[nedge] (a1)--(a0);
    \end{scope}
  \end{tikzpicture}\right)^2\nonumber\\
&=q_\alpha^{K_2}\left(\left(
\begin{tikzpicture}[baseline=-.8ex]
  \tfrootededge
    \draw[nedge] (a2)--(a0);
    \draw[nedge] (a1)--(a0);
    font=\normalsize
    \node at (1.8*\unit,0) {$-\xi\times$};
  \begin{scope}[xshift=3.5*\unit]
    \tfrootededge
    \draw[edge] (a1)--(a0);
    \draw[nedge] (a2)--(a0);
    \end{scope}
  \node at (5.4*\unit,0) {$-\xi\times$};
  \begin{scope}[xshift=7.1*\unit]
      \tfrootededge
    \draw[edge] (a2)--(a0);
    \draw[nedge] (a1)--(a0);
    \end{scope}
  \end{tikzpicture}\right)^2\right).
\label{sec3}
\end{align}
Hence,
\begin{equation}
\begin{split}
0&\le q_\alpha\left(\unlab{\left(%
\begin{tikzpicture}[baseline=-.8ex]
  \tfrootededge
    \draw[nedge] (a2)--(a0);
    \draw[nedge] (a1)--(a0);
    font=\normalsize
    \node at (1.8*\unit,0) {$-\xi\times$};
  \begin{scope}[xshift=3.5*\unit]
    \tfrootededge
    \draw[edge] (a1)--(a0);
    \draw[nedge] (a2)--(a0);
    \end{scope}
  \node at (5.4*\unit,0) {$-\xi\times$};
  \begin{scope}[xshift=7.1*\unit]
      \tfrootededge
    \draw[edge] (a2)--(a0);
    \draw[nedge] (a1)--(a0);
    \end{scope}
  \end{tikzpicture}\right)^2}{K_2}\right)\\
&= q_\alpha\left(\left\llbracket
\begin{tikzpicture}[baseline=-.8ex]
    \ffrootededge
  \draw[nedge] (a3)--(a0);
  \draw[nedge] (a0)--(a2);
  \draw[nedge] (a1)--(a0);
  \draw[nedge] (a1)--(a3);
  \draw[nedge] (a2)--(a1);
    font=\normalsize
\node at (1.2*\unit,0) {$+$};
  \begin{scope}[xshift=2.4*\unit]
    \ffrootededge
    \draw[edge] (a1)--(a0);
  \draw[nedge] (a3)--(a0);
  \draw[nedge] (a2)--(a0);
  \draw[nedge] (a1)--(a3);
  \draw[nedge] (a2)--(a1);
    \end{scope}
  \node at (4.4*\unit,0) {$+\xi^2\cdot\bigg($};
  \begin{scope}[xshift=6.5*\unit]
      \ffrootededge
    \draw[edge] (a3)--(a0);
    \draw[edge] (a3)--(a1);
    \draw[nedge] (a1)--(a0);
    \draw[nedge] (a2)--(a0);
  \draw[nedge] (a2)--(a1);
    \end{scope}
\node at (7.7*\unit,0) {$+$};
      \begin{scope}[xshift=8.9*\unit]
      \ffrootededge
    \draw[edge] (a2)--(a0);
    \draw[edge] (a2)--(a1);
    \draw[nedge] (a1)--(a0);
    \draw[nedge] (a3)--(a0);
    \draw[nedge] (a3)--(a1);
    \end{scope}
    \node at (10.1*\unit,0) {$+$};
      \begin{scope}[xshift=11.3*\unit]
      \ffrootededge
    \draw[nedge] (a2)--(a0);
    \draw[nedge] (a2)--(a1);
    \draw[edge] (a1)--(a0);
    \draw[edge] (a3)--(a0);
    \draw[edge] (a3)--(a1);
    \end{scope}
    \node at (12.5*\unit,0) {$+$};
      \begin{scope}[xshift=13.7*\unit]
      \ffrootededge
    \draw[edge] (a2)--(a0);
    \draw[edge] (a2)--(a1);
    \draw[edge] (a1)--(a0);
    \draw[nedge] (a3)--(a0);
    \draw[nedge] (a3)--(a1);
    \end{scope}
    \node at (14.9*\unit,0) {$+$};
      \begin{scope}[xshift=16.1*\unit]
      \ffrootededge
      \draw[nedge] (a2)--(a0);
    \draw[edge] (a2)--(a1);
    \draw[nedge] (a1)--(a0);
    \draw[edge] (a3)--(a0);
    \draw[nedge] (a3)--(a1);
\end{scope}
    \node at (17.3*\unit,0) {$+$};
      \begin{scope}[xshift=18.5*\unit]
      \ffrootededge
    \draw[nedge] (a2)--(a0);
    \draw[edge] (a2)--(a1);
    \draw[edge] (a1)--(a0);
    \draw[edge] (a3)--(a0);
    \draw[nedge] (a3)--(a1);
    \end{scope}
    \node at (19.7*\unit,0) {$\bigg)$};
   \end{tikzpicture}\right.\right.\\
&\quad\quad\quad\quad
   \left.\left.\begin{tikzpicture}[baseline=-.8ex]
    \node at (0*\unit,0) {$-\xi\cdot\bigg($};
      \begin{scope}[xshift=2*\unit]
      \ffrootededge
    \draw[nedge] (a2)--(a0);
    \draw[nedge] (a2)--(a1);
    \draw[nedge] (a1)--(a0);
    \draw[nedge] (a3)--(a0);
    \draw[edge] (a3)--(a1);
    \end{scope}
    \node at (3.2*\unit,0) {$+$};
      \begin{scope}[xshift=4.4*\unit]
      \ffrootededge
    \draw[nedge] (a2)--(a0);
    \draw[nedge] (a2)--(a1);
    \draw[edge] (a1)--(a0);
    \draw[nedge] (a3)--(a0);
    \draw[edge] (a3)--(a1);
    \end{scope}
    \node at (5.6*\unit,0) {$+$};
      \begin{scope}[xshift=6.8*\unit]
      \ffrootededge
    \draw[edge] (a2)--(a0);
    \draw[nedge] (a2)--(a1);
    \draw[nedge] (a1)--(a0);
    \draw[nedge] (a3)--(a0);
    \draw[nedge] (a3)--(a1);
    \end{scope}
    \node at (8*\unit,0) {$+$};
      \begin{scope}[xshift=9.2*\unit]
      \ffrootededge
    \draw[edge] (a2)--(a0);
    \draw[nedge] (a2)--(a1);
    \draw[edge] (a1)--(a0);
    \draw[nedge] (a3)--(a0);
    \draw[nedge] (a3)--(a1);
    \end{scope}
    \node at (10.4*\unit,0) {$\bigg)$};
     \end{tikzpicture}\right\rrbracket_{K_2}\right).
     \end{split}
\label{sec4}
\end{equation}
Evaluating the operator $\unlab{\cdot}{K_2}$ yields the following inequality.
\begin{equation}
0\le q_\alpha\left(0,\frac{1}{6},\frac{1}{3},\frac{-\xi}{3},0,
              \frac{\xi^2-2\xi}{6},\frac{\xi^2}{2},\frac{2\xi^2}{3},\frac{\xi^2}{6},0,
	      0\right).
\label{secD}
\end{equation}
As an example of the evaluation, consider the third coordinate: the only four-vertex graph
with two non-incident edges appears with the coefficient one in the sum. The probability that
a randomly chosen pair of vertices in the four-vertex graph formed by two non-incident edges
yields this term of the sum is $1/3$ which is the third coordinate of the final vector.

Now, let us sum the equations and inequalities~\eqref{secAA}, \eqref{secA}, \eqref{secB}, \eqref{secC} and~\eqref{secD}
with coefficients $\frac{3\beta}{\sqrt{1+8\beta}}$,
$\frac{3}{4}\cdot\left(3-\frac{5+8\beta}{\sqrt{1+8\beta}}\right)$, $\frac{3}{\sqrt{1+8\beta}}$, $\frac{3}{\sqrt{1+8\beta}}$ and
$\frac{3}{4}\cdot\left(1+\frac{1+4\beta}{\sqrt{1+8\beta}}\right)$, respectively, and
substitute $\xi=\frac{\sqrt{1+8\beta}-1}{2\beta}-1$.
Note that the coefficients for the inequalities~\eqref{secC} and~\eqref{secD} are non-negative.
So, we eventually deduce that
\begin{multline}\label{sec0}
\frac{3}{4}\beta(3-\sqrt{8\beta+1})\le
      q_\alpha\Bigg(0,\frac{1}{2},1-\frac{1}{\sqrt{1+8\beta}},\frac{1}{2},\frac{9}{8}-\frac{3+12\beta}{8\sqrt{1+8\beta}},
\frac{1}{2},\\
0,0,\frac{9}{8}-\frac{15+12\beta}{8\sqrt{1+8\beta}},\frac{15}{8}-\frac{21+20\beta}{8\sqrt{1+8\beta}},\frac{9}{4}-\frac{15+12\beta}{4\sqrt{1+8\beta}}\Bigg).
%
\end{multline}
Finally, since $q_\alpha$ is positive, we derive from~\eqref{sec0} (the fifth, ninth, tenth and eleventh coordinates
are decreasing for $\beta\in [0,1]$ since their derivates are positive in this range) that
\begin{equation}
\frac{3}{4}\beta(3-\sqrt{8\beta+1})\le
 q_\alpha\left(0,\frac{1}{2},1,\frac{1}{2},1,\frac{1}{2},0,0,\frac{1}{2},\frac{1}{2},1\right)=q_\alpha\left(\overline{P_3}+K_3\right).
\label{sec-final}
\end{equation}

Observe that the function $x\mapsto\frac{3}{4}x(3-\sqrt{8x+1})$ is
increasing on the interval $\icc{0}{2/9}$ and
that
\[\frac{3}{4}x(3-\sqrt{8x+1})\ge 2/9=\frac{3}{4}\cdot\frac{2}{9}\left(3-\sqrt{8\cdot2/9+1}\right)\]
for $x\in\icc{2/9}{1/2}$. Hence, the left hand side of~\eqref{sec-final} is at least $\frac{3}{4}\alpha(3-\sqrt{8\alpha+1})$
for $\alpha\in\icc{0}{2/9}$ as asserted in the statement of the theorem.
\end{proof}

\section{Conclusion}
Using more sophisticated methods, we have been able to further improve
the bounds on $\varphi_2(\alpha)$. However, the proof becomes extremely
complicated and since we have not been able to prove that
\[\varphi_2(\alpha)=\frac{3\alpha(1+\sqrt{1-4\alpha})}{4},\]
which is the bound given by the best known example, we have decided
not to further pursue our work in this direction.
To show the limits of our current approach, let us mention that
Theorem~\ref{thm:main} asserts that $\varphi_2(1/12)\ge 0.10681$ and
we can push the bound to $\varphi_2(1/12)\ge 0.11099$;
the simple bound is $0.08333$ and the expected bound is $0.11353$
for this value.

We have also attempted together with Andrzej Grzesik to apply
this method for improving bounds on $\varphi_3$. Though we have
been able to obtain some improvements, e.g., we can show that
$\varphi_3(1/20)\ge 0.05183$, the level of technicality of
the argument seems to be too large for us to be able to report
on our findings in an accessible way at this point.

\bigskip
\textbf{Acknowledgment.} The authors are grateful to Ji\v r\' i Matou\v sek
for drawing to their attention the manuscript on the subject
that he coauthored with Uli Wagner~\cite{MaWa11}. In particular, they thank him for
pointing out their notion of pagoda and its applications.

\end{document}